\documentclass{article}

\usepackage{amssymb,amsmath,color}
\usepackage{amsthm}
\usepackage{url}
\usepackage{tikz,subcaption}
\usepackage[enableskew]{youngtab}
\usepackage{ytableau,varwidth}

\definecolor{red}{rgb}{1,0,0}
\definecolor{blue}{rgb}{.2,.2,.8}

\newtheorem{theorem}{Theorem}[section]
\newtheorem{corollary}[theorem]{Corollary}

\theoremstyle{definition}

\begin{document}

	\title{Overpartitions and functions from multiplicative number theory}
	\author{Mircea Merca
	\\ 
	\footnotesize Department of Mathematics, University of Craiova, 200585 Craiova, Romania\\
	\footnotesize mircea.merca@profinfo.edu.ro
}
	\date{}
	\maketitle


\begin{abstract}
Let $\alpha$ and $\beta$ be two nonnegative integers such that $\beta < \alpha$.
For an arbitrary sequence $\{a_n\}_{n\geqslant 1}$ of complex numbers, we consider the generalized Lambert series in order to investigate linear combinations of the form
$\sum_{k\geqslant 1} S(\alpha k-\beta,n) a_k$, where $S(k,n)$ is the total number of non-overlined parts equal to $k$ in all the overpartitions of $n$. The general nature of the numbers $a_n$ allows us to provide connections between 
overpartitions and functions from multiplicative number theory. 
\end{abstract}

\section{Introduction} 

	A partition of a positive integer $n$ is a sequence of positive integers whose sum is $n$. The order of the summands is unimportant when writing the partitions of $n$, but for consistency, a partition of $n$ will be written with the summands in a nonincreasing order \cite{Andrews98}.
	
An overpartition of a positive integer $n$ is a partition of $n$ in which the first occurrence of a part
of each size may be overlined or not \cite{Corteel}. Let $\overline{p}(n)$ denote the number of overpartitions of an integer $n$. For example, $\overline{p}(3)=8$ because  there are $8$ overpartitions of $3$:
$$(3),\ (\bar{3}),\ (2,1),\ (\bar{2},1),\ (2,\bar{1}),\ (\bar{2},\ \bar{1}),\ (1,1,1),\ (\bar{1},1,1).$$
Since the overlined parts form a partition into distinct parts and the non-overlined parts form an ordinary partition, the  generating function for overpartitions is given by
$$
\sum_{n=0}^\infty \overline p(n) q^n = \frac{(-q;q)_\infty}{(q;q)_\infty}.
$$
Here and throughout this paper, we use the following customary $q$-series notation:
\begin{align*}
& (a;q)_n = \begin{cases}
1, & \text{for $n=0$,}\\
(1-a)(1-aq)\cdots(1-aq^{n-1}), &\text{for $n>0$;}
\end{cases}\\
& (a;q)_\infty = \lim_{n\to\infty} (a;q)_n.
\end{align*}
Because the infinite product $(a;q)_{\infty}$ diverges when $a\neq 0$ and $|q| \geqslant 1$, whenever
$(a;q)_{\infty}$ appears in a formula, we shall assume $|q| < 1$.

The following theta identity is often attributed to Gauss \cite[p. 23, eqs. (2.2.12)]{Andrews98}:
\begin{equation}\label{eq:Gauss}
1+2\sum_{n=1}^{\infty} (-1)^n q^{n^2} = \frac {(q;q)_\infty} {(-q;q)_\infty}.
\end{equation}
Rewriting this expression as 
\begin{equation*}
\frac {(-q;q)_\infty} {(q;q)_\infty} \left( 1+2\sum_{n=1}^{\infty} (-1)^n q^{n^2} \right)=1 ,
\end{equation*}
we deduce the following linear recurrence relation for the overpartition function $\overline{p}(n)$:
\begin{equation}\label{Rec}
\overline{p}(n) + 2 \sum_{j=1}^{\infty} (-1)^j \overline{p}(n-j^2)=\delta_{0,n},
\end{equation}
where $\delta_{i,j}$ is the Kronecker delta function and $\overline{p}(n)=0$ if $n$ is a negative integer. 

We denote by $S(k,n)$ the number of  non-overlined parts equal to $k$ in all the overpartitions of $n$. For example,
the overpartitions of $4$ are:
$(4)$, $(\overline{4})$, 
$(3,1)$, $(\overline{3},1)$, $(3,\overline{1})$, $(\overline{3},\overline{1})$,
$(2,2)$, $(\overline{2},2)$, 
$(2,1,1)$, $(\overline{2},1,1)$, $(2,\overline{1},1)$, $(\overline{2},\overline{1},1)$, 
$(1,1,1,1)$ and $(\overline{1},1,1,1)$.
Then we have
$S(1,4) = 15$, $S(2,4) = 5$, $S(3,4) = 2$, and $S(4,4) = 1$.

Let $\alpha$ and $\beta$ be two integer such that $0\leqslant \beta < \alpha$.
For an arbitrary sequence $\{a_n\}_{n\geqslant 1}$ of complex numbers we consider
$$A(a,\alpha,\beta;n) = \sum_{k=1}^\infty S(\alpha k-\beta,n) a_k.$$
For example, if $a_n=1$ for any positive integer $n$, then  $A(1,\alpha,\beta;n)$ counts the non-overlined parts congruent to $-\beta$ modulo $\alpha$ in all the overpartitions of $n$. On the other hand, we have:
\begin{align*}
& A(a,1,0;4) = 15a_1+5a_2+2a_3+a_4,\\
& A(a,2,0;4) = 5a_1+a_2,\\
& A(a,2,1;4) = 15a_1+2a_2.
\end{align*}

Also for $\alpha, \beta \in \mathbb{Z}$ such that $0\leqslant \beta < \alpha$, and  an arbitrary sequence $\{a_n\}_{n\geqslant 1}$ of complex numbers, we consider generalized Lambert series expansions of the form \cite{MrkS}
\begin{equation}\label{LS}
\sum_{n=1}^\infty \frac{a_n q^{\alpha n - \beta}}{1-q^{\alpha n - \beta}} = \sum_{n=1}^\infty B(a,\alpha,\beta;n) q^n, \qquad |q|<1.
\end{equation}
The coefficients of the generalized Lambert series expansion on the left-hand side of
the previous equation are given on the right by
$$B(a,\alpha,\beta;n) = \sum_{\alpha d-\beta|n} a_d.$$
Clearly this sum runs over all positive divisors of $n$ which are congruent to $-\beta$ modulo $\alpha$.
If $a_n=1$ for any positive integer $n$, then  $B(1,\alpha,\beta;n)$ counts the positive divisors of $n$ which are congruent to $-\beta$ modulo $\alpha$.

In \cite{Andrews18}, Andrews and Merca 
defined $\overline{M}_k(n)$ to be the number of overpartitions of $n$ in which the first part larger than $k$ appears at least $k+1$ times. 
If $n=12$ and $k=2$ then we have $M_2(12)=16$ because the overpartitions in question are 
$(4,4,4)$,
$(\overline{4},4,4)$,
$(3,3,3,3)$,
$(\overline{3},3,3,3)$,
$(3,3,3,2,1)$,
$(3,3,3,\overline{2},1)$,
$(3,3,3,2,\overline{1})$,
$(3,3,3,\overline{2},\overline{1})$,
$(\overline{3},3,3,2,1)$,
$(\overline{3},3,3,\overline{2},1)$,
$(\overline{3},3,3,2,\overline{1})$,
$(\overline{3},3,3,\overline{2},\overline{1})$,
$(3,3,3,1,1,1)$,
$(3,3,3,\overline{1},1,1)$,
$(\overline{3},3,3,1,1,1)$,
$(\overline{3},3,3,\overline{1},1,1)$.

In this article, we provide the following identity.

\begin{theorem}\label{TH2}
	Let $\alpha, k, n$ be positive integers and let $\beta$ be a nonnegative integer such that $\beta < \alpha$.
    For an arbitrary sequence $\{a_m\}_{m\geqslant 1}$ of complex numbers, we have
	\begin{align*}
	& (-1)^{k} \left(A(a,\alpha,\beta;n)+ 2\sum_{j=1}^k (-1)^j A(a,\alpha,\beta;n-j^2) - B(a,\alpha,\beta;n) \right) \\
	& = \sum_{j=1}^n B(a,\alpha,\beta;j)\overline{M}_k(n-j).
	\end{align*}
\end{theorem}

An immediate consequence of Theorem \ref{TH2} is given by the following infinite
family of linear inequalities.

\begin{corollary}\label{C3}
	Let $\alpha, k, n$ be positive integers and let $\beta$ be a nonnegative integer such that $\beta < \alpha$.
	For an arbitrary sequence $\{a_m\}_{m\geqslant 1}$ of nonnegative real numbers, we have
	\begin{equation*}
	(-1)^{k} \left(A(a,\alpha,\beta;n)+ 2\sum_{j=1}^k (-1)^j A(a,\alpha,\beta;n-j^2) - B(a,\alpha,\beta;n) \right) \geqslant 0, 
	\end{equation*}
	with strict inequality if $n\geqslant (k+1)^2$. For example:
	\begin{align*}
	& A(a,\alpha,\beta;n)-2A(a,\alpha,\beta;n-1) \leqslant B(a,\alpha,\beta;n),\\
	& A(a,\alpha,\beta;n)-2A(a,\alpha,\beta;n-1)+2A(a,\alpha,\beta;n-4) \geqslant B(a,\alpha,\beta;n).
	\end{align*}
\end{corollary}

	On the other hand, the limiting case $k\to\infty$ of Theorem \ref{TH2} reads as follows.
	
	\begin{corollary}\label{TH1}
		Let $\alpha, \beta, n$ be nonnegative integers such that $\beta < \alpha$.
		For an arbitrary sequence $\{a_m\}_{m\geqslant 1}$ of complex numbers, we have
		$$
		A(a,\alpha,\beta;n)+2\sum_{j=1}^\infty (-1)^j A\big(a,\alpha,\beta;n-j^2\big) = B(a,\alpha,\beta;n).
		$$
	\end{corollary}

	The classical M\"{o}bius function $\mu(n)$ is an important multiplicative function in number theory and combinatorics. This function is defined for all positive integers $n$ and has its values in $\{-1, 0, 1\}$ depending on the factorization of $n$ into prime factors:
	$$
	\mu(n)=
	\begin{cases}
	0, & \text{if $n$ has a squared prime factor,}\\
	(-1)^k, & \text{if $n$ is a product of $k$ distinct primes.}
	\end{cases}
	$$
	Considering the identity \cite[Theorem 263]{Hardy79}
	$$
	\sum_{n=1}^{\infty} \frac{\mu(n) q^{n}}{1-q^{n}} = q,
	$$
	and the expression \eqref{GFA} for the generating function of $A(a,\alpha,\beta;n)$, we deduce that
	$$\sum_{n=1}^\infty A(\mu,1,0;n) q^n = \frac{(-q;q)_\infty}{(q;q)_\infty}\cdot q = \sum_{n=1}^\infty \overline{p}(n-1) q^n.$$
	For $(a,\alpha,\beta)=(\mu,1,0)$, we see that the statement of Corollary \ref{TH1} reduces to the recurrence relation \eqref{Rec}.
Moreover, we remark that the overpartition function $\overline{p}(n)$ can be expressed in terms of the classical 
M\"{o}bius function $\mu(n)$ as follows:
$$ \overline{p}(n) = \sum_{k=1}^{n+1} S(k,n+1) \mu(k).$$
For example, the case $n=3$ of this relation provides the following decomposition of $\overline{p}(3)$:
$$S(1,4)-S(2,4)-S(3,4)=15-5-2=8.$$

The general nature of the sequence $\{a_n\}_{n\geqslant 1}$ allows for applications of Theorems \ref{TH2} 
to many important functions from multiplicative number theory.
The ordinary Lambert series expansions studied in 
\cite{MERCA-LSFACTTHM,MERCA-SCHMIDT1,
	SCHMIDT-LSFACTTHM} 
can be converted into a series of the form in \eqref{LS} as follows:
$$
\sum_{n=1}^{\infty} \frac{f(n) q^{\alpha n - \beta}}{1-q^{\alpha n - \beta}} = 
\sum_{n=1}^{\infty} \left( \sum_{\alpha d-\beta|n} f(d)\right) q^n,
$$
where $f$ is any of the following functions:
the \textit{M\"{o}bius function} $\mu(n)$, 
\textit{Euler's phi function} $\phi(n)$, 
the generalized \emph{sum of divisors function} $\sigma_{\alpha}(n)$,
\textit{Liouville's function} $\lambda(n)$,
\textit{von Mangoldt's function}  $\Lambda(n)$,
and \textit{Jordan's totient function} $J_t(n)$.
In this way, many identities of the following form
\begin{equation}\label{F}
\sum_{k=1}^{\lfloor n/\alpha \rfloor} S(\alpha k-\beta,n) f(k)
= \sum_{k=1}^n B(f,\alpha,\beta;k) \overline{p}(n-k)
\end{equation}
can be easily derived.
These relations are important since there are
rarely such simple and universal identities expressing formulas for an entire
class of special arithmetic functions considered in the context of so many
applications in number theory and combinatorics.

The results proved in this article continue the spirit of \cite{Merca15,Merca16,Merca17, MrkMJOM,MrkMJOM1,Merca18a,Merca18b,MrkS,SCHMIDT-LSFACTTHM} by connecting the seemingly disparate branches of additive and multiplicative number theory. For further reading at the intersection of the additive and multiplicative branches of number theory, we recommend \cite{Alladi}, \cite[Ch. VII]{Cai}, \cite{Jameson,Schneider,Wakhare}.

The paper is organized as follows. We will first prove Theorem \ref{TH2} in Section \ref{S2}. 
In Section \ref{S3}, we consider the case $(a,\alpha,\beta)=(\phi,1,0)$ of Theorem \ref{TH2} and remark three relations involving the sum of non-overlined parts counted without multiplicity, in all the overpartitions of $n$. 
In Section \ref{S4}, we show that the number of distinct prime divisors of $n$ can be expressed in terms of the number of prime non-overlined  parts in all the overpartitions of $n$.
In the last section, as a third application of our results, we provide connections between the number of unitary divisors of $n$ and  the number of squarefree non-overlined parts in all the overpartitions of $n$.

\section{Proof of Theorem \ref{TH2}}
\label{S2}

The bivariate generating function for overpartitions where $z$ keeps track of the number of  non-overlined parts equal to $k$ is given by
$$
\frac{1-q^k}{1-zq^k}\cdot \frac{(-q;q)_\infty}{(q;q)_\infty}.
$$
For $k>0$, we deduce that
\begin{align*}
\sum_{n=0}^\infty S(k,n) q^n =  \frac{d}{dz}\Bigr|_{z=1} \frac{(1-q^k)(-q;q)_\infty}{ (1-zq^k) (q;q)_\infty}=
\frac{q^k}{1-q^k}\cdot \frac{(-q;q)_\infty}{(q;q)_\infty}
\end{align*}
is the generating function for the number of $k$'s in all the overpartitions of $n$. We can write
\begin{align*}
\sum_{k=1}^\infty \left( \sum_{n=1}^\infty S(\alpha k-\beta,n) q^n \right) a_k 
= \frac{(-q;q)_\infty}{(q;q)_\infty} \sum_{k=1}^\infty \frac{a_k q^{\alpha k - \beta}}{1-q^{\alpha k - \beta}}.
\end{align*}
Thus we deduce that the generating function for $A(a,\alpha,\beta;n)$ is given by
\begin{align}
\sum_{n=1}^\infty A(a,\alpha,\beta;n) q^n = \frac{(-q;q)_\infty}{(q;q)_\infty} \sum_{n=1}^\infty \frac{a_n q^{\alpha n - \beta}}{1-q^{\alpha n - \beta}}.\label{GFA}
\end{align}

	In \cite{Andrews18}, the authors considered the theta identity \eqref{eq:Gauss} 
	and proved the following truncated form:
	\begin{align}
& \frac{(-q;q)_{\infty}} {(q;q)_{\infty}} \left(1 + 2 \sum_{j=1}^{k} (-1)^j q^{j^2} \right) \label{eq4} \\
& \qquad = 1+2 (-1)^k \frac{(-q;q)_k}{(q;q)_k} \sum_{j=0}^{\infty} \frac{q^{(k+1)(k+j+1)}(-q^{k+j+2};q)_{\infty}}{(1-q^{k+j+1})(q^{k+j+2};q)_{\infty}}.\notag
\end{align}

Multiplying both sides of \eqref{eq4} by \eqref{LS},
we obtain
\begin{align*}
& (-1)^{k} \left( \bigg( \sum_{n=1}^\infty A(a,\alpha,\beta;n) q^n \bigg) \bigg(1+2  \sum_{n=1}^{k} (-1)^{n} q^{n^2}\bigg) -\sum_{n=1}^\infty B(a,\alpha,\beta;n)q^{n}\right)   \\
& =   \left( \sum_{n=1}^\infty B(a,\alpha,\beta;n) q^{n} \right) \left( \sum_{n=0}^\infty \overline{M}_k(n) q^n\right),
\end{align*}
where we have invoked 
the generating function for $\overline{M}_k(n)$ \cite{Andrews18},
$$
\sum_{n=0}^\infty \overline{M}_k(n) q^n = 
\frac{2(-q;q)_k}{(q;q)_k} \sum_{j=0}^{\infty} \frac{q^{(k+1)(k+j+1)}(-q^{k+j+2};q)_{\infty}}{(1-q^{k+j+1})(q^{k+j+2};q)_{\infty}}.
$$
The proof follows easily considering Cauchy's multiplication of two power series.

\section{Euler's totient and overpartitions}
\label{S3}

Euler's totient or phi function, $\phi(n)$, is a multiplicative function that counts the totatives of $n$, that is the positive integers less than or equal to $n$ that are relatively prime to $n$. According to Euler's classical formula \cite[Theorem 63]{Hardy79}, we have
$$B(\phi,1,0;n) = n.$$
In addition, by \eqref{GFA} we deduce that
\begin{equation}\label{eq5}
\sum_{n=1}^\infty A(\phi,1,0;n) q^n = \frac{(-q;q)_\infty}{(q;q)_\infty} \sum_{n=1}^\infty nq^n = \frac{(-q;q)_\infty}{(q;q)_\infty}\cdot \frac{q}{ (1-q)^2}.
\end{equation}
Moreover, we have the following combinatorial interpretation of \eqref{eq5}.

\begin{corollary}
	For $n\geqslant 0$, the sum of non-overlined parts counted without multiplicity, in all the overpartitions of $n$ equals
	$$\sum_{j=1}^n j \overline{p}(n-j).$$
\end{corollary}

\begin{proof}
	The bivariate generating function for overpartitions where $z$ keeps track of parts which are not overlined without multiplicity is
	\begin{align*}
	& (-q;q)_\infty \prod_{j=1}^\infty\big(1+z^{j}(q^{j}+q^{2j}+q^{3j}+\cdots) \big) 
	 = \frac{(-q;q)_\infty}{(q;q)_\infty} \prod_{j=1}^\infty \big(1+(z^{j}-1) q^{j} \big).
	\end{align*} 
	Applying $\frac{d}{dz}\Bigr|_{z=1}$
	to this equation, we obtain the expression of the generating function for the sum  of non-overlined parts counted without multiplicity, in all the overpartitions of $n$:
	$$
	\frac{(-q;q)_\infty}{(q;q)_\infty} \sum_{j=1}^\infty jq^{j} 
	= \frac{(-q;q)_\infty}{(q;q)_\infty}\cdot \frac{q}{ (1-q)^2}.
	$$
\end{proof}

The following result is a consequence of Theorem \ref{TH2}.

\begin{corollary}\label{C61}
	For $k,n>0$, the sum of non-overlined parts counted without multiplicity, in all the overpartitions of $n$ satisfies the following identity:
	\begin{align*}
	& (-1)^{k} \left(A(\phi,1,0;n)+2 \sum_{j=1}^k (-1)^j A\big(\phi,1,0;n-j^2\big) - n \right) 
	= \sum_{j=1}^n j \overline{M}_k(n-j).
	\end{align*}
\end{corollary}

Relevant to this corollary, it would be very appealing to have a combinatorial interpretation of
$$\sum_{j=1}^n j \overline{M}_k(n-j).$$

The limiting case $k\to\infty$ of Corollary \ref{C61} reads as follows.

\begin{corollary}\label{C6.2}
	For $n>0$, the sum of non-overlined parts counted without multiplicity, in all the overpartitions of $n$ satisfies the following linear recurrence relation:
	\begin{align*}
	& A(\phi,1,0;n)+2 \sum_{j=1}^\infty (-1)^j A\big(\phi,1,0;n-j^2\big) = n.
	\end{align*}
\end{corollary}

We remark that $A(\phi,1,0;n)$ and $n$ have the same parity.

\section{Prime parts in overpartitions}
\label{S4}

In this section, we consider the characteristic function of the prime numbers, namely
$$
\chi(n)=
\begin{cases}
0, & \text{if $n$ is composite,}\\
1, & \text{if $n$ is prime.}
\end{cases}
$$
It is clear that 
$B(\chi,\alpha,\beta;n)$
counts the prime divisors of $n$ which are congruent to $-\beta$ modulo $\alpha$, 
while $A(\chi,\alpha,\beta;n)$ counts in all the overpartitions of $n$ the prime parts of the form $\alpha k -\beta$, where $k$ is a prime.

On the other hand, it is clear that $A(\chi,1,0;n)$ counts the prime non-overlined parts in all the overpartitions of $n$ and
$$B(\chi,1,0;n)=\omega(n)$$ 
where $\omega(n)$ is a classical additive function defined as the number of distinct primes dividing $n$.

We remark the following results.

\begin{corollary}
	For $n\geqslant 0$, the number of non-overlined prime parts in all the overpartitions of $n$ equals
	$$\sum_{j=2}^n \omega(j) \overline{p}(n-j).$$
\end{corollary}

For example, the overpartitions of $5$ are:
$(5)$, $(\overline{5})$, 
$(4,1)$, $(\overline{4},1)$, $(4,\overline{1})$, $(\overline{4},\overline{1})$, 
$(3,2)$, $(\overline{3},2)$, $(3,\overline{2})$, $(\overline{3},\overline{2})$, 
$(3,1,1)$, $(\overline{3},1,1)$, $(3,\overline{1},1)$, $(\overline{3},\overline{1},1)$, 
$(2,2,1)$, $(\overline{2},2,1)$, $(2,2,\overline{1})$, $(\overline{2},2,\overline{1})$,
$(2,1,1,1)$, $(\overline{2},1,1,1)$, $(2,\overline{1},1,1)$, $(\overline{2},\overline{1},1,1)$, 
$(1,1,1,1,1)$ and $(\overline{1},1,1,1,1)$.
The number of prime parts which are not overlined  in all the overpartitions of $5$ is equal to 
$1+0+0+0+0+0+2+1+1+0+1+0+1+0+2+1+2+1+1+0+1+0+0+0=15$. On the other hand, we have
$\overline{p}(3)+\overline{p}(2)+\overline{p}(1)+\overline{p}(0) = 8+4+2+1=15$.

\begin{corollary}\label{C4.2}
	For $k,n>0$, the number of prime non-overlined parts in all the overpartitions of $n$ and the number of distinct prime divisors of $n$ satisfies the identity:
	\begin{align*}
	& (-1)^{k} \left(A(\chi,1,0;n)+2 \sum_{j=1}^k (-1)^j A\big(\chi,1,0;n-j^2\big) - \omega(n) \right) \\
	& = \sum_{j=2}^n \omega(j) \overline{M}_k(n-j).
	\end{align*}
\end{corollary}

The limiting case $k\to\infty$ of this corollary provides the following decomposition of $\omega(n)$.

\begin{corollary}
	For $n>0$, 
	\begin{align*}
	A(\chi,1,0;n)+2 \sum_{j=1}^k (-1)^j A\big(\chi,1,0;n-j^2\big) = \omega(n) .
	\end{align*}
\end{corollary}

Similar results can be obtained for the sum of distinct prime divisors of $n$ and the sum of prime non-overlined parts in all the overpartitions of $n$ if we replace the function $\chi(n)$ by  $n\cdot \chi(n)$.

Related to Corollary \ref{C4.2}, it is still an open problem to give partition-theoretic interpretations for
$$\sum_{j=2}^n \omega(j) \overline{M}_k(n-j).$$

\section{Squarefree parts in overpartitions}

Taking into account that
$$
|\mu(n)|=
\begin{cases}
0, & \text{if $n$ has a squared prime factor,}\\
1, & \text{otherwise,}
\end{cases}
$$
we deduce that $B(|\mu|,\alpha,\beta;n)$ counts the divisors of $n$ of the form $\alpha d -\beta$ where $d$ is squarefree.
In addition, $A(|\mu|,\alpha,\beta;n)$ counts in all the overpartitions of $n$ the parts which are not overlined and are of the form  $\alpha k -\beta$, where $k$ is a squarefree.

Recall that a natural number $d$ is a \textit{unitary divisor} of a number $n$ if $d$ is a divisor of $n$ and if $d$ and $n/d$ are coprime.
The sum over all the positive divisors of $n$ of the absolute value of the M\"{o}bius function is equal to the number of unitary divisors of $n$ \cite[Theorem 264]{Hardy79},
$$B(|\mu|,1,0;n) =2^{\omega(n)}.$$
On the other hand, it is clear that $B(|\mu|,1,0;n)$ counts  the squarefree divisors of $n$, while $A(|\mu|,1,0;n)$ counts the squarefree  non-overlined parts in all the overpartitions of $n$. 

We have the following particular results.

\begin{corollary}
	For $n\geqslant 0$, the number of squarefree non-overlined parts in all the overpartitions of $n$ equals
	$$\sum_{j=1}^n 2^{\omega(j)} \overline{p}(n-j).$$
\end{corollary}

For example, the number of squarefree non-overlined parts in all the overpartitions of $5$ is equal to 
$1+0+1+1+0+0+2+1+1+0+3+2+2+1+3+2+2+1+4+3+3+2+5+4=44$. On the other hand, we have
$\overline{p}(4)+2\overline{p}(3)+2\overline{p}(2)+2\overline{p}(1)+2\overline{p}(0) = 14+16+8+4+2=44$.

\begin{corollary}
	For $k,n>0$, the number of squarefree non-overlined parts in all the overpartitions of $n$ and the number of unitary divisors of $n$ satisfies the identity:
	\begin{align*}
	& (-1)^{k} \left(A(|\mu|,1,0;n)+2 \sum_{j=1}^k (-1)^j A\big(|\mu|,1,0;n-j^2\big) - 2^{\omega(n)} \right) \\
	& = \sum_{j=1}^n 2^{\omega(j)} \overline{M}_k(n-j).
	\end{align*}
\end{corollary}

On the other hand, we have
$$\sum_{d|n} (-1)^{1+d} |\mu(d)| =
\begin{cases}
2^{\omega(n)}, &\text{for $n$ odd,}\\
0, & \text{for $n$ even,}
\end{cases}
$$
while 
$$A\big((-1)^{1+d}|\mu(d)|,1,0;n\big)=\sum_{k=1}^n (-1)^{k+1} S(k,n) |\mu(k)|$$ 
is the difference between the number of squarefree  non-overlined odd parts and the number of squarefree  non-overlined even parts in all the overpartitions of $n$. 
For example, the number of squarefree  non-overlined odd parts in all the overpartitions of $5$ is equal to $34$, while  
the number of squarefree  non-overlined even parts in all the overpartitions of $5$ equals $10$. We see that
$$S(1,5)-S(2,5)+S(3,5)+S(5,5) = 29 - 10 + 4 + 1 = 24$$
equals the difference $34-10$.
We provide new connections between squarefree non-overlined parts and unitary divisors.

\begin{corollary}
	Let $k$ be a positive integer. For $n$ odd, the difference between the number of squarefree  non-overlined odd parts and the number of squarefree  non-overlined even parts in all the overpartitions of $n$ satisfies the following two identity:
	\begin{align*}
	& (-1)^{k} \left(A\big((-1)^{1+d}|\mu(d)|,1,0;n\big) + 2\sum_{j=1}^k (-1)^j A\big((-1)^{1+d}|\mu(d)|,1,0;n-j^2\big) - 2^{\omega(n)} \right) \\
	& = \sum_{j=1}^{(n+1)/2} 2^{\omega(2j-1)} \overline{M}_k(n-2j+1).
	\end{align*}
\end{corollary}

\begin{corollary}
	Let $k$ be a positive integer. For $n$ even, the difference between the number of squarefree  non-overlined odd parts and the number of squarefree  non-overlined even parts in all the overpartitions of $n$ satisfies the following identity:
\begin{align*}
& (-1)^{k} \left(A\big((-1)^{1+d}|\mu(d)|,1,0;n\big) + 2\sum_{j=1}^k (-1)^j A\big((-1)^{1+d}|\mu(d)|,1,0;n-j^2\big) \right) \\
& = \sum_{j=1}^{(n+1)/2} 2^{\omega(2j-1)} \overline{M}_k(n-2j+1).
\end{align*}
\end{corollary}

In addition, considering  the relation \cite[Exercise 1.52]{McCarthy}
$$\sum_{d|n}2^{\omega(d)}=\sigma_0(n^2),$$
where $\sigma_0(n)$ counts the positive divisors of $n$, we can write the following identity.

\begin{corollary}
	For $k,n>0$, 
	\begin{align*}
	& (-1)^{k} \left(A(2^{\omega(n)},1,0;n)+2 \sum_{j=1}^k (-1)^j A\big(2^{\omega(n)},1,0;n-j^2\big) - \sigma_0(n^2) \right) \\
	& = \sum_{j=1}^n \sigma_0(j^2) \overline{M}_k(n-j).
	\end{align*}
\end{corollary}

Relevant to this corollary, it would be very appealing to have combinatorial interpretations of $$A(2^{\omega(n)},1,0;n)=\sum_{j=1}^n \sigma_0(j^2) \overline{p}(n-j)\qquad\text{and}\qquad \sum_{j=1}^n \sigma_0(j^2) \overline{M}_k(n-j).$$

\bigskip


\end{document}